\documentclass[12pt,reqno]{article}

\usepackage[usenames]{color}
\usepackage{amssymb}
\usepackage{amscd}

\usepackage[colorlinks=true,
linkcolor=webgreen,
filecolor=webbrown,
citecolor=webgreen]{hyperref}

\definecolor{webgreen}{rgb}{0,.5,0}
\definecolor{webbrown}{rgb}{.6,0,0}

\usepackage{color}
\usepackage{fullpage}
\usepackage{float}

\usepackage{graphics,amsmath,amssymb}
\numberwithin{equation}{section}
\usepackage{amsthm}
\usepackage{amsfonts}
\usepackage{latexsym}

\DeclareMathOperator{\arctanh}{arctanh}
\DeclareMathOperator{\arccot}{arccot}
\DeclareMathOperator{\arccoth}{arccoth}
\DeclareMathOperator{\sgn}{sgn}

\begin{document}

\theoremstyle{plain}
\newtheorem{theorem}{Theorem}
\newtheorem{corollary}[theorem]{Corollary}
\newtheorem{lemma}[theorem]{Lemma}
\theoremstyle{remark}
\newtheorem{remark}[theorem]{Remark}
\newtheorem{example}[theorem]{Example}

\allowdisplaybreaks[4]

\begin{center}
\vskip 1cm{\LARGE\bf On Some Series Involving the Central Binomial Coefficients\\
\vskip .11in }

\vskip 0.8cm

{\large
Kunle Adegoke \\
Department of Physics and Engineering Physics, \\ Obafemi Awolowo University, Ile-Ife\\ Nigeria \\
\href{mailto:adegoke00@gmail.com}{\tt adegoke00@gmail.com}

\vskip 0.25 in

Robert Frontczak \\ Independent Researcher, Reutlingen\\ Germany \\
\href{mailto:robert.frontczak@lbbw.de}{\tt robert.frontczak@web.de}

\vskip 0.25 in

Taras Goy \\
Faculty of Mathematics and Computer Science \\ Vasyl Stefanyk Precarpathian National University, Ivano-Frankivsk\\ Ukra\-ine \\
\href{mailto:taras.goy@pnu.edu.ua}{\tt taras.goy@pnu.edu.ua}}
\end{center}

\vskip .25in

\begin{abstract}

In this paper, we explore a variety of series involving the central binomial coefficients, highlighting their structural properties and connections to other mathematical objects. Specifically, we derive new closed-form representations and examine the convergence properties of infinite series with a repeating alternation pattern of signs involving central binomial coefficients. More concretely, we derive the series  $$\sum\limits_{n=0}^{\infty}\frac{(-1)^{\omega_n}}{2n+1}\tbinom{2n}{n}x^n,\,\,\, \sum\limits_{n=0}^{\infty}{(-1)^{\omega_n}}\tbinom{2n}{n}x^n\,\,\, \text{and} \,\,\, \sum\limits_{n=0}^{\infty}{(-1)^{\omega_n}}n\tbinom{2n}{n}x^n,$$
where $\omega_n$ represents both $\lfloor\frac{n}{2}\rfloor$ and
$\lceil\frac{n}{2}\rceil$. Also, we present novel series involving Fibonacci and Lucas numbers, deriving many interesting identities.

\medskip


\medskip


\end{abstract}

\section{Introduction}

In the literature, the terms ``central binomial coefficients'' or ``middle binomial coefficients'' are typically used to refer to the binomial coefficients  $$\binom{2n}{n}=\frac{(2n)!}{(n!)^2},$$ as they form the central spine of Pascal's triangle when it is viewed as a centrally symmetric triangle. The central binomial coefficients have long been a cornerstone in combinatorics and mathematical analysis due to their rich combinatorial interpretations and connections to diverse areas of mathematics. These coefficients appear in a variety of contexts, ranging from the enumeration of lattice paths and Catalan structures to their role in number theory, special functions, and asymptotic analysis \cite{alzer,banderier,barry,ferrari,ford,kessler,tauraso}.

One particularly intriguing aspect of central binomial coefficients is their occurrence in both finite and infinite series. Such series naturally emerge in problems involving combinatorial summation, generating functions, and integral representations of special functions, often unveiling deeper connections between seemingly unrelated branches of mathematics by bridging combinatorial structures with analytical tools \cite{batir1,batir2,campbell,lehmer,li1,li2,lim,pomer,qi,sprug}.

It is worth noting that, since $$\binom{2n}{n}=(n+1)C_{n},$$ where $C_{n}$ denotes the Catalan numbers, our results can equivalently be expressed in terms of the Catalan numbers. Similar series have been recently studied by the authors \cite{AFG1,AFG2}, as well as in the works \cite{bhandari,boya,chen,chu}, among others.

For more information about central binomial coefficients and their applications  we refer to the  On-Line Encyclopedia of Integer Sequences \cite{OEIS} where central binomial coefficients are indexed under entry A000984.

Despite their ubiquity, the study of series involving central binomial coefficients is far from exhausted. Many of their structural properties and relationships remain underexplored, and new techniques continue to uncover unexpected results. In this paper, we contribute to this ongoing exploration by deriving and analyzing various series involving central binomial coefficients and floor and ceiling functions. Similar classes of series were studied recently by Fan and Chu \cite{FanChu} for the Riemann zeta and the Dirichlet beta functions.

In Section 3 we will also establish connections with the Fibonacci and Lucas numbers. As usual, the Fibonacci numbers $F_n$ and the Lucas numbers $L_n$ are defined
through the recurrence relations $F_n = F_{n-1}+F_{n-2}$, $n\ge 2$, with initial values $F_0=0$, $F_1=1$ and $L_n = L_{n-1}+L_{n-2}$, $n\ge 2$, with $L_0=2$, $L_1=1$. For negative indices, the recurrence relations are given by  $F_{-n}=(-1)^{n-1}F_n$ and $L_{-n}=(-1)^n L_n$.
For any integer $n$, they possess  Binet's formulas
\begin{equation}\label{binet}
F_n = \frac{\alpha^n - \beta^n }{\alpha - \beta },\qquad L_n = \alpha^n + \beta^n,
\end{equation}
with $\alpha=(1+\sqrt{5})/2$ being the golden ratio and $\beta=-{1}/{\alpha}$.

\section{Main results}

This section is based on the following fundamental lemma.
\begin{lemma}\label{main_lem}
	For real variable $x$, we have
	\begin{gather}\label{LemmaRe}
	{\rm Re}\big(\arcsin \big((1+i)x\big) \big) = \arctan\!\left(\frac{\sqrt2x}{\sqrt{\sqrt {1 + 4x^4 } + 1}}\right)\!,\\
	\label{LemmaIm}
	{\rm Im}\big(\arcsin \big((1+i)x\big) \big)= \arctanh\!\left(\frac{\sqrt2x}{\sqrt{\sqrt {1 + 4x^4 } + 1}}\right)\!,
	\end{gather}
	where $i=\sqrt{-1}$ is the imaginary unit.
\end{lemma}
\begin{proof}
	Both formulas can be easily derived from the identity
	\begin{equation*}
	\arcsin \big((1+i)x\big) = \arctan\! \left(\!\frac{x\sqrt2}{\sqrt{\sqrt {1 + 4x^4 }- 1}}\right) +
	\arctan\!\left(\!\frac{i x\sqrt2}{\sqrt{\sqrt {1 + 4x^4 }- 1}}\right)\!,
	\end{equation*}
	which follows by taking the tangent of both sides and applying the well-known formulas
	\begin{gather*}
	\tan(\arcsin x)=\frac{x}{\sqrt{1-x^2}},\qquad \tan(x+y)=\frac{\tan x+\tan y}{1-\tan x\tan y},\\
	\arctan x= -i\arctanh(ix),
	\end{gather*}
	 and performing some algebraic manipulations.
\end{proof}
\begin{lemma} \label{diff}
	We have
	\begin{gather*}
	\frac{d}{dx}\arctan\! \left( {\frac{x}{\sqrt {1 + \sqrt {1 + x^4 } } }} \right) = \sqrt {\frac{{\sqrt {1 + x^4 } - x^2 }}{2(1 + x^4 )}} ,\\
	\frac{d}{dx}\arctanh\! \left( {\frac{x}{{\sqrt {1 + \sqrt {1 + x^4 } } }}} \right) = \sqrt {\frac{{\sqrt {1 + x^4 } + x^2 }}{2( {1 + x^4 } )}} ,\\
	\frac{d^2}{dx^2}\arctan\! \left( {\frac{x}{{\sqrt {1 + \sqrt {1 + x^4 } } }}} \right) = - x\big( \sqrt {1 + x^4 } +2x^2 \, \big)\sqrt {\frac{{\sqrt {1 + x^4 } - x^2 }}{{2(1 + x^4)^3 }}},
	\end{gather*} 
	\begin{equation*}
	\frac{d^2}{dx^2}\arctanh\! \left( {\frac{x}{{\sqrt {1 + \sqrt {1 + x^4 } } }}} \right) = x\big( \sqrt {1 + x^4}-2x^2\big)\sqrt {\frac{{\sqrt {1 + x^4 }  + x^2 }}{2( 1 + x^4)^3 }};
	\end{equation*}
	so that
	\begin{gather*}
	\frac{d}{dx}\arctan\! \left( {\frac{x}{{\sqrt {1 + \sqrt {1 + x^4 } } }}} \right)\!\frac{d}{dx}\arctanh\! \left( \frac{x}{\sqrt {1 + \sqrt {1 + x^4 }}} \right)\! = \frac{1}{2(1 + x^4)},\\
	\frac{d^2}{dx^2}\arctan\! \left( {\frac{x}{{\sqrt {1 + \sqrt {1 + x^4 } } }}} \right)\!\frac{d^2}{dx^2}\arctanh\! \left( {\frac{x}{{\sqrt {1 + \sqrt {1 + x^4 } } }}} \right)\! = \frac{{x^2 ( {3x^4 - 1})}}{{2( {1 + x^4 })^3 }}.
	\end{gather*}
\end{lemma}

In the following theorem, by using Maclaurin expansion of $\arcsin x$, we derive some series involving central binomial coefficients in numerators.
\begin{theorem}\label{main_thm}
	For $|x|\leq 1$, the following series identities hold:
	\begin{gather}
	\sum_{n = 0}^\infty {\frac{{( - 1)^{\left\lceil \frac{n}2 \right\rceil } }\binom{{2n}}{n}}{(2n+1)2^{2n}}x^{2n + 1} }
	= \sqrt 2 \arctan\!\left( {\frac{x}{{\sqrt {1 + \sqrt {1 + x^4 } } }}} \right)\!,\label{New_1_1} \\
	\sum_{n = 0}^\infty {\frac{{( - 1)^{\left\lfloor \frac{n}2 \right\rfloor } }\binom{{2n}}{n}}{(2n+1)2^{2n}} x^{2n + 1} }
	= \sqrt 2 \arctanh\! \left( {\frac{x}{{\sqrt {1 + \sqrt {1 + x^4 } } }}} \right)\!,\label{New_1_2} \\
	\sum_{n = 0}^\infty {\frac{( - 1)^{\left\lceil \frac{n}2 \right\rceil }\binom{{2n}}{n}}{2^{2n}} x^n }
	= \sqrt {\frac{{\sqrt {1 + x^2 } - x}}{{{1 + x^2 }}}} ,\label{eq.e6dkxz3} \\
	\sum_{n = 0}^\infty {\frac{( - 1)^{\left\lfloor\frac{n}2 \right\rfloor }\binom{{2n}}{n}}{2^{2n}} x^n }
	= \sqrt {\frac{{\sqrt {1 + x^2 } + x}}{{ {1 + x^2 }}}},\label{eq.gktxw1u} \\
	\sum_{n = 1}^\infty {\frac{( - 1)^{\left\lceil\frac{n}2\right\rceil }\binom{{2n}}{n}n}{2^{2n}}  x^n }
	= - \frac x2\big( {\sqrt {1 + x^2 } + 2x } \big)\sqrt {\frac{{\sqrt {1 + x^2 } - x}}{{( {1 + x^2 })^3 }}},
	\end{gather}
	and
	\begin{equation}\label{eq.cu2e27v}
	\sum_{n = 1}^\infty {\frac{( - 1)^{\left\lfloor \frac{n}2 \right\rfloor }\binom{{2n}}{n}n}{2^{2n}} x^n }
	= \frac x2\big( { \sqrt {1 + x^2 } - 2x } \big)\sqrt {\frac{{\sqrt {1 + x^2 } + x}}{{( {1 + x^2 })^3 }}}.
	\end{equation}
\end{theorem}
\begin{proof}
	Consider the Maclaurin expansion of $\arcsin x$:
	\begin{equation}\label{arcsin}
	\arcsin x= \sum_{n = 0}^\infty \frac{ \binom{2n}n}{2^{2n}(2n + 1)}x^{2n + 1} ,\quad |x|\leq1.
	\end{equation}
	Write $\sqrt2(1+i)x$ for $x$ in \eqref{arcsin} and take the real and imaginary parts using \eqref{LemmaRe}, \eqref{LemmaIm} and
	$${\rm Re}\big((1+i)^{2n+1}\big)=(-1)^{\lceil\frac{n}{2}\rceil}2^n, \quad {\rm Im}\big((1+i)^{2n+1}\big)=(-1)^{\lfloor\frac{n}{2}\rfloor}2^n.
	$$
	This leads to \eqref{New_1_1} and \eqref{New_1_2}. The remaining identities follow from \eqref{New_1_1} or \eqref{New_1_2} by differentiation, applying  Lemma \ref{diff}.
\end{proof}

The series \eqref{eq.e6dkxz3}--\eqref{eq.cu2e27v} allow us to give new proofs for some known combinatorial identities.
\begin{corollary}
	If $n$ is a non-negative integer, then
	\begin{gather}
	\sum_{k = 0}^n {\binom{{2k}}{k}\binom{{2( {n - k} )}}{{n - k}}}  = 2^{2n} \label{eq.u9rel73},\\
	\sum_{k = 1}^{n - 1} {\binom{{2k}}{k}\binom{{2( {n - k})}}{{n - k}}k( {n - k})} = n( {n - 1})2^{2n-3}\label{eq.dw9s201} ,
	\end{gather}
	while if $n$ is an odd non-negative integer, then
	\begin{gather}
	\sum_{k = 0}^n {( - 1)^k \binom{{2k}}{k}\binom{2( {n - k})}{n - k}} = 0 \label{eq.depf1dv},\\
	\sum_{k = 1}^{n - 1} {(-1)^k\binom{2k}{k}\binom{{2( {n - k} )}}{{n - k}}k( {n - k} )} =0 \label{eq.gtr2q3o}.
	\end{gather}
\end{corollary}
\begin{proof}
	Identity \eqref{eq.e6dkxz3} multiplied by identity \eqref{eq.gktxw1u} gives
	\begin{equation*}
	\sum_{n = 0}^\infty {\frac{( - 1)^{\left\lceil \frac{n}2 \right\rceil }}{2^{2n}} \binom{{2n}}{n} x^n } \sum_{n = 0}^\infty  {\frac{( - 1)^{\left\lfloor \frac{n}2 \right\rfloor }}{2^{2n}} \binom{{2n}}{n} x^n } = \frac{1}{{1 + x^2 }} = \sum_{n = 0}^\infty {( - 1)^n x^{2n} },
	\end{equation*}
	so that an application of the Cauchy product rule gives
	\begin{equation*}
	\sum_{n = 0}^\infty  { \frac{x^n}{2^{2n}} \sum_{k = 0}^n {( - 1)^{\left\lceil\frac {k}2 \right\rceil  + \left\lfloor {\frac{n - k}2} \right\rfloor } \binom{{2k}}{k}\binom{{2( {n - k})}}{{n - k}}} }  = \sum_{n = 0}^\infty  {( - 1)^n x^{2n}};
	\end{equation*}
	and hence
	\begin{align*}
	&\sum_{n = 0}^\infty { \frac{x^{2n}}{2^{4n}} \sum_{k = 0}^{2n} {( - 1)^{\left\lceil \frac{k}2 \right\rceil + \left\lfloor {\frac{2n - k}2}\right\rfloor } \binom{{2k}}{k}\binom{{2( {2n - k} )}}{{2n - k}}} }   \\
	&\qquad\qquad+\sum_{n = 0}^\infty  { \frac{x^{2n - 1}}{2^{ 2(2n - 1)}} \sum_{k = 0}^{2n - 1} {( - 1)^{\left\lceil \frac{k}2 \right\rceil  + \left\lfloor {\frac{2n - 1 - k}2} \right\rfloor } \binom{{2k}}{k}\binom{{2( {2n - 1 - k} )}}{{2n - 1 - k}}} } = \sum\limits_{n = 0}^\infty  {( - 1)^n x^{2n} }.
	\end{align*}
	Comparing the coefficients of $x^n$ on the left-hand side and the right-hand side gives \eqref{eq.u9rel73} and \eqref{eq.depf1dv}.
	The proof of \eqref{eq.dw9s201} and \eqref{eq.gtr2q3o} is similar and follows from
	\begin{align*}
	\sum_{n = 1}^\infty  & \frac{( - 1)^{\left\lceil \frac{n}2 \right\rceil }\binom{{2n}}{n}n}{2^{2n}}  x^n \sum_{n = 1}^\infty  \frac{( - 1)^{\left\lfloor \frac{n}2 \right\rfloor }\binom{{2n}}{n}n}{2^{2n}} x^n  =\frac{{x^2( 3x^2  - 1 )}}{4( 1 + x^2 )^3 } \\
	&\qquad\quad= \frac{3}{{4( 1 + x^2 )}} - \frac{7}{4( {1 + x^2 })^2} + \frac{1}{( {1 + x^2 })^3} = \frac14\sum\limits_{n = 1}^\infty  {( - 1)^n {n( {2n - 1} )}}x^{2n}.
	\end{align*}
	Note that we used
	$\lceil \frac{k}2\rceil  + \lfloor - \frac{k}2\rfloor  = 0$ and $\lceil \frac{k}2 \rceil  - \left\lceil \frac{k + 1}2 \right\rceil  = ( - 1)^k$.
\end{proof}
\begin{remark}
	The formula in \eqref{eq.u9rel73} is well known; see, for example, \cite[p.~32]{gould}, \cite[p.~187]{graham} or the article by Miki\'{c} \cite{Mikic}. Formula \eqref{eq.depf1dv} can be found in \cite[p.~33]{gould} regardless of the parity of $n$ as follows:
		\begin{equation*}
	\sum_{k = 0}^n (- 1)^k \binom{{2k}}{k} \binom{2(n - k)}{n - k} = \frac{1+(-1)^n}{2} \binom{n}{n/2} 2^n
		\end{equation*}
	or in an equivalent form in \cite[Eq.~(2)]{Mikic}. 
	
Formula \eqref{eq.dw9s201} can be derived from the relations
	\begin{gather*}
	\sum_{k = 0}^{n} k \binom{{2k}}{k} \binom{2(n - k)}{n - k} = n 2^{2n-1},\\
	\sum_{k = 0}^{n-1} k^2 \binom{{2k}}{k} \binom{2(n-k)}{n - k} = n (3n+1) 2^{2n-3},
	\end{gather*}
	both being deductions from
	\begin{equation*}
	\sum_{k=0}^n 2^{2(n-k)} \binom{n}{k} \binom{2k}{k} t^k = \sum_{k=0}^n \binom{2k}{k} \binom{2(n-k)}{n-k} (1+t)^k,
	\end{equation*}
	which can also be found in Gould's book \cite{gould}.
\end{remark}

We continue with some examples that are consequences of the main results stated in Theorem \ref{main_thm}.
\begin{example}
If $x=1/2$, $x={\sqrt2}/4$, and $x=1/4$ then from Theorem \ref{main_thm} we have
	\begin{align*}
	&\sum_{n = 0}^\infty \frac{(- 1)^{\left\lceil{\frac{n}{2}}\right\rceil}\binom{2n}{n}}{2^{2n}(2n+1)} =\sqrt2\arccot(\sqrt{\delta}),\quad \sum_{n = 0}^\infty \frac{(- 1)^{\left\lfloor{\frac{n}2}\right\rfloor}\binom{2n}{n}}{2^{2n}(2n+1)} =\sqrt2\arccoth(\sqrt{\delta}),\\
	&\sum_{n = 0}^\infty \frac{(- 1)^{\left\lceil{\frac{n}{2}}\right\rceil}\binom{2n}{n}}{2^{3n}(2n+1)} =2\arccot(\sqrt{\alpha^{3}}),\quad\, \sum_{n = 0}^\infty \frac{(- 1)^{\left\lfloor{\frac{n}2}\right\rfloor}\binom{2n}{n}}{2^{3n}(2n+1)} =2\arccoth(\sqrt{\alpha^{3}}),\\
	&\hspace{3.1cm}\sum_{n = 0}^\infty \frac{(- 1)^{\left\lceil\frac{n}{2}\right\rceil}\binom{2n}{n}}{2^{4n}(2n+1)} =2\sqrt2\arctan{\sqrt{\sqrt{17}-4}},\\
	&\hspace{3.1cm}\sum_{n = 0}^\infty \frac{(- 1)^{\left\lfloor{\frac{n}{2}}\right\rfloor}\binom{2n}{n}}{2^{4n}(2n+1)} =2\sqrt2\arctanh{\sqrt{\sqrt{17}-4}},
	\end{align*}
	\vspace{-0.3 cm}
	\begin{align*}
	&\sum_{n = 0}^\infty \frac{(- 1)^{\left\lceil{\frac{n}{2}}\right\rceil}\binom{2n}{n}}{2^{2n}} =\frac1{\sqrt{2\delta}},\hspace{2.675cm} \sum_{n = 0}^\infty \frac{(- 1)^{\left\lfloor{\frac{n}{2}}\right\rfloor}\binom{2n}{n}}{2^{2n}} =\frac{\sqrt{2\delta}}{2},\\
	&\sum_{n = 0}^\infty \frac{(- 1)^{\left\lceil{\frac{n}{2}}\right\rceil}\binom{2n}{n}}{2^{3n}} =\frac{2}{\sqrt{5\alpha}},\hspace{2.65cm} \sum_{n = 0}^\infty \frac{(- 1)^{\left\lfloor{\frac{n}{2} }\right\rfloor}\binom{2n}{n}}{2^{3n}}=\frac{2\sqrt{5\alpha}}{5},	\\
	&\sum_{n = 0}^\infty \frac{(- 1)^{\left\lceil{\frac{n}{2}}\right\rceil}\binom{2n}{n} }{2^{4n}} =\frac{2\sqrt{\sqrt{17}-1}}{\sqrt{17}},\hspace{1.45cm}\sum_{n = 0}^\infty \frac{(- 1)^{\left\lfloor{\frac{n}{2}}\right\rfloor}\binom{2n}{n} }{2^{4n}} =\frac{2\sqrt{\sqrt{17}+1}}{\sqrt{17}},
	\end{align*}
	and
	\begin{align*}
	&\sum_{n = 1}^\infty \frac{(- 1)^{\left\lceil{\frac{n}{2}}\right\rceil}n\binom{2n}{n}}{2^{2n}} =-\frac{\sqrt{\delta}}{4},\hspace{1.6cm} \sum_{n = 1}^\infty \frac{(- 1)^{\left\lfloor{\frac{n}{2}}\right\rfloor}n\binom{2n}{n}}{2^{2n}} =-\frac{1}{4\sqrt{\delta}},\\
	&\sum_{n = 1}^\infty \frac{(- 1)^{\left\lceil{\frac{n}{2}}\right\rceil}n\binom{2n}{n}}{2^{3n}} =-\frac{\sqrt{5\alpha^5}}{25},\hspace{1.2cm} \sum_{n = 1}^\infty \frac{(- 1)^{\left\lfloor{\frac{n}{2}}\right\rfloor}n\binom{2n}{n}}{2^{3n}} =\frac{1}{5\sqrt{5\alpha^5}},\\
	&\hspace{2cm}\sum_{n = 1}^\infty \frac{(- 1)^{\left\lceil{\frac{n}{2}}\right\rceil}n\binom{2n}{n}}{2^{4n}} =-\frac{\sqrt{17}}{289}\sqrt{17\sqrt{17}+47},\\
	&\hspace{2cm}\sum_{n = 1}^\infty \frac{(- 1)^{\left\lfloor{\frac{n}{2}}\right\rfloor}n\binom{2n}{n}}{2^{4n}} =\frac{\sqrt{17}}{289}\sqrt{17\sqrt{17}- 47},
	\end{align*}
	where  $\delta=1+\sqrt2$ denotes the silver ratio.
\end{example}

By setting $x=\frac12 \sqrt{\tan\varphi}$ and performing some algebraic manipulations, we derive  the trigonometric versions of Theorem \ref{main_thm}.
\begin{theorem}\label{trigo}
	For $|\varphi|\leq\frac{\pi}{4}$, the following series identity hold true:
	\begin{gather*}
	\sum_{n = 0}^\infty \frac{(- 1)^{\left\lceil{\frac{n}{2}}\right\rceil}\binom{2n}{n}}{(2n+1)2^{2n}} \tan^n\varphi= \sgn\varphi \sqrt{2\cot\varphi}\arctan\sqrt{\tan\frac{\varphi}{2}},\\
	\sum_{n = 0}^\infty \frac{(- 1)^{\left\lfloor{\frac{n}{2}}\right\rfloor}\binom{2n}{n}}{(2n+1)2^{2n}}\tan^n\varphi= \sgn\varphi \sqrt{2\cot\varphi}\arctanh\sqrt{\tan\frac{\varphi}{2}}	,\\
	\sum_{n = 0}^\infty \frac{(- 1)^{\left\lceil{\frac{n}{2}}\right\rceil}\binom{2n}{n}}{2^{2n}}\tan^n\varphi= \sqrt{2\cos\varphi}\cos\!\left(\frac{\varphi}{2}+\frac{\pi}{4}\right)\!,\\
	\sum_{n = 0}^\infty \frac{(- 1)^{\left\lfloor{\frac{n}{2}}\right\rfloor}\binom{2n}{n}}{2^{2n}}\tan^n\varphi= \sqrt{2\cos\varphi}\cos\!\left(\frac{\varphi}{2}-\frac{\pi}{4}\right)\!,
	\end{gather*}
	and
	\begin{gather*}
	\sum_{n = 1}^\infty\frac{ (- 1)^{\left\lceil{\frac{n}{2}}\right\rceil}n\binom{2n}{n} }{2^{2n}}\tan^{n}\varphi = \frac{1+2\sin\varphi}{\sqrt{2\cos\varphi}}\sin2\varphi\sin\!\Big(\frac{\varphi}{2}-\frac{\pi}{4}\Big),\\
	\sum_{n = 1}^\infty \frac{(- 1)^{\left\lfloor{\frac{n}{2}}\right\rfloor}n\binom{2n}{n}}{2^{2n}}\tan^{n}\varphi  = \frac{\cos\frac\varphi 2\sin^2\!\big(\frac{\varphi}{2}+\frac{\pi}{4}\big)}{\sqrt{\cos\varphi}}
	\big(\cos 2\varphi -\sin 2\varphi + \cos\varphi +3 \sin\varphi-2\big).
	\end{gather*}
\end{theorem}
\begin{example}
	Choosing $\varphi={\pi}/{6}$ and $\varphi={\pi}/{8}$ in the formulas of Theorem \ref{trigo} with $(-1)^{\lceil{{n}/{2}}\rceil}$
	in conjunction with the trigonometric evaluations
	\begin{gather*}
	\sin\frac{\pi}{8} = \sqrt{\frac{2-\sqrt2}{2}}, \qquad \cos\frac{\pi}{8} = \sqrt{\frac{2+\sqrt2}{2}},	\qquad
	\sin\frac{\pi}{16} = \frac{\sqrt{2-\sqrt{2+\sqrt2}}}{2},\\
	\cos\frac{\pi}{16} = \frac{\sqrt{2+\sqrt{2+\sqrt2}}}{2}, \qquad
	\sin\frac{\pi}{12} = \frac{\sqrt{2-\sqrt3}}{2}, \qquad \cos\frac{\pi}{12} = \frac{\sqrt{2+\sqrt3}}{2},
	\end{gather*}
	yields the following results:
	\begin{gather*}
	\sum_{n = 0}^\infty \frac{(- 1)^{\left\lceil{\frac{n}{2}}\right\rceil}\binom{2n}{n}}{(2n+1)(4\sqrt3)^n}={\sqrt{2\sqrt3}}\arctan{\sqrt{2-\sqrt3}},\\
	\sum_{n = 0}^\infty \frac{(- 1)^{\left\lceil{\frac{n}{2}}\right\rceil}\binom{2n}{n}}{(4\sqrt3)^n}=\frac{\sqrt{\sqrt3}}{2},\hspace{1.3cm}
	\sum_{n = 0}^\infty \frac{(- 1)^{\left\lceil{\frac{n}{2}}\right\rceil}n\binom{2n}{n}}{(4\sqrt3)^n}=-\frac{\sqrt{\sqrt3}}{4};\\
	\sum_{n = 0}^\infty \frac{(- 1)^{\left\lceil{\frac{n}{2}}\right\rceil}\binom{2n}{n}}{(2n+1)(4\delta)^n}=\sqrt{2\delta}\arctan\sqrt{\sqrt{\sqrt{8}\delta}-\delta},\\
	\sum_{n = 0}^\infty \frac{(- 1)^{\left\lceil{\frac{n}{2}}\right\rceil}\binom{2n}{n}}{(4\delta)^n}=\frac{\sqrt{\delta\sqrt{\sqrt8\delta}}}{\sqrt2\,\Big(\sqrt{2+\sqrt{\sqrt{2}\delta}}+\sqrt{2-\sqrt{\sqrt{2}\delta}}\Big)},\\
	\sum_{n = 0}^\infty \frac{(- 1)^{\left\lceil{\frac{n}{2}}\right\rceil}n\binom{2n}{n}}{(4\delta)^n} = -\frac{\sqrt{2\delta}+\sqrt{\sqrt8}}{4\sqrt{\sqrt{\sqrt8}\delta}\,\Big(\sqrt{2+\sqrt{\sqrt{2}\delta}}+\sqrt{2-\sqrt{\sqrt{2}\delta}}\Big)}.
	\end{gather*}
\end{example}

By adding and subtracting \eqref{New_1_1} and \eqref{New_1_2}, one can derive series involving the binomial coefficients $\binom{4n}{2n}$ and $\binom{4n+2}{2n+1}$. Similar series can also be obtained using other pairs of formulas from Theorem \ref{main_thm}. Alternating series with binomial coefficients $\binom{4n}{2n}$ and $\binom{4n+2}{2n+1}$ in the denominator were recently studied by the authors in \cite{NNTDM-24}.
\begin{corollary} For $|x|\leq\frac12$, the following series identity hold:
	\begin{align*}
	\sum_{n = 0}^\infty&   \frac{(-1)^n\binom{4n}{2n}}{4n+1}x^{4n+1}= \frac{\sqrt2}{4}\left(\arctanh\!\left(\frac{2x}{\sqrt{1+\sqrt{1+16x^4}}}\right)+\arctan\!\left(\frac{2	x}{\sqrt{1+\sqrt{1+16x^4}}}\right)\!\right)\!,\\
	\sum_{n = 0}^\infty&   \frac{(-1)^n\binom{4n+2}{2n+1}}{4n+3}x^{4n+3}= \frac{\sqrt2}{4}\left(\arctanh\!\left(\frac{2x}{\sqrt{1+\sqrt{1+16x^4}}}\right)-\arctan\!\left(\frac{2	x}{\sqrt{1+\sqrt{1+16x^4}}}\right)\!\right)\!.
	\end{align*}
\end{corollary}

\section{Series involving Fibonacci and Lucas numbers}

In this section, we present a family of Fibonacci and  Lucas series identities involving central binomial coefficient and  additional integer parameters $m$ and $s$.
\begin{theorem} \label{mn+s} 
	For any integers $m,s$ and real $p\geq4\alpha^m$,
	\begin{align}
	\sum_{n = 0}^\infty & 
	\frac{(-1)^{\left\lceil{\frac{n}{2}}\right\rceil}\binom{2n}{n}}{p^n(2n+1)} F_{mn+s} = \sqrt{\frac{p}{10}}
	\left(\alpha^{s-\frac{m}2}\arctan h(\alpha)  -\beta^{s-\frac{m}2}\arctan h(\beta)\right)\!,\label{FL_11}\\
\sum_{n = 0}^\infty&\frac{(- 1)^{\left\lceil{\frac{n}{2}}\right\rceil}\binom{2n}{n}}{p^n(2n+1)} L_{mn+s} = \sqrt{\frac{p}{2}}
	\left(\alpha^{s-\frac{m}2}\arctan h(\alpha)  +\beta^{s-\frac{m}2}\arctan h(\beta)\right)\!,\label{FL_21}\\
	\sum_{n = 0}^\infty& \frac{(- 1)^{\left\lfloor{\frac{n}{2}}\right\rfloor}\binom{2n}{n}}{p^n(2n+1)} F_{mn+s}=\sqrt{\frac{p}{10}}
	\left(\alpha^{s-\frac{m}2}\arctanh h(\alpha)  -\beta^{s-\frac{m}2}\arctanh h(\beta)\right)\!,\\
	\sum_{n = 0}^\infty& \frac{(- 1)^{\left\lfloor{\frac{n}{2}}\right\rfloor}\binom{2n}{n}}{p^n(2n+1)}  L_{mn+s} = \sqrt{\frac{p}{2}}
	\left(\alpha^{s-\frac{m}2}\arctanh h(\alpha)  +\beta^{s-\frac{m}2}\arctanh h(\beta)\right)\!,\label{18}
	\end{align}
where 
	$$ h(z)=\frac{2\sqrt{z^m}}{\sqrt{\sqrt{p^2+16z^{2m}}+p}};$$
	\begin{gather}
	\label{19}
	\sum_{n = 0}^\infty \frac{(- 1)^{\left\lceil\frac{n}{2}\right\rceil}\binom{2n}{n}}{p^n}  F_{mn+s}
	= \sqrt\frac{p}{5}\,
	\big(\alpha^s r^{-}(\alpha)-\beta^s r^{-}(\beta)\big),\\
	\label{20}
	\sum_{n = 0}^\infty \frac{(- 1)^{\left\lceil\frac{n}{2}\right\rceil}\binom{2n}{n}}{p^n} L_{mn+s}
	= \sqrt{p}\,\big(\alpha^s r^{-}(\alpha)+\beta^s r^{-}(\beta)\big),\\
	\label{21}
	\sum_{n = 0}^\infty \frac{(- 1)^{\left\lfloor\frac{n}{2}\right\rfloor}\binom{2n}{n}}{p^n}  F_{mn+s}
	= \sqrt{\frac{p}{5}}\,
	\big(\alpha^sr^{+}(\alpha)-\beta^s r^{+}(\beta)\big),\\
	\label{22}
	\sum_{n = 0}^\infty \frac{(- 1)^{\left\lfloor\frac{n}{2}\right\rfloor}\binom{2n}{n}}{p^n}  L_{mn+s}
	= \sqrt{p}\,	\big(\alpha^sr^{+}(\alpha)+\beta^s r^{+}(\beta)\big),
	\end{gather}
	where $$ r^{\pm}(z)=
	\sqrt{\frac{{\sqrt{p^2+16z^{2m}}\pm 4z^m}}{p^2+16z^{2m}}};$$
	and
	\begin{gather*} 
	\sum_{n = 1}^\infty \frac{(- 1)^{\left\lfloor{\frac{n}{2}}\right\rfloor}n\binom{2n}{n}}{p^n} F_{mn+s}=-\frac{2\sqrt5}{5}\sqrt{p}\left(\alpha^{m+s}t^{-}(\alpha)-\beta^{m+s}t^{-}(\beta)\right)\!,\\
	\sum_{n = 1}^\infty \frac{(- 1)^{\left\lfloor{\frac{n}{2}}\right\rfloor}n\binom{2n}{n}}{p^{n}} L_{mn+s}= 
	-2\sqrt{p}\left(\alpha^{m+s}t^{-}(\alpha)+\beta^{m+s}t^{-}(\beta)\right)\!,\\
	\sum_{n = 1}^\infty \frac{(- 1)^{\left\lceil\frac{n}2\right\rceil}n\binom{2n}{n}}{p^{n}} F_{mn+s}= - \frac{2\sqrt5}{5}\sqrt p \left(\alpha^{m+s}t^{+}(\alpha)-\beta^{m+s}t^{+}(\beta)\right)\!,\\
	\sum_{n = 1}^\infty \frac{(- 1)^{\left\lceil{\frac{n}{2}}\right\rceil}n\binom{2n}{n}} {p^{n}} L_{mn+s}=-2\sqrt{p}\left(\alpha^{m+s}t^{+}(\alpha)+\beta^{m+s}t^{+}(\beta)\right)\!,
	\end{gather*}
	where
	$$	t^{\pm}(z) = \big(8z^m \mp \sqrt{p^2+16z^{2m}}\big)\sqrt{\frac{\sqrt{p^2+16z^{2m}}\pm4z^m}{(p^2+16z^{2m})^3}}.$$
\end{theorem} 
\begin{proof} To prove \eqref{FL_11} and \eqref{FL_21}, insert $x=4\alpha^m/{p}$ and $x=4\beta^m/p$, in turn, into \eqref{New_1_1}.  Then multiply through by $\alpha^s$ and $\beta^s$, respectively, and combine the following resulting expressions 
	\begin{gather*} 
	\sum_{n = 0}^\infty \frac{(- 1)^{\left\lceil{\frac{n}{2}}\right\rceil}\binom{2n}{n}}{p^n(2n+1)} \alpha^{mn+s}
	= \sqrt{\frac{p}{2}}\,\alpha^{s-\frac m2}
	\arctan\left(\frac{2\sqrt{\alpha^m}}{\sqrt{p^2+16\alpha^{2m}}+p}\right)\!,\\
	\sum_{n = 0}^\infty \frac{(- 1)^{\left\lceil{\frac{n}{2}}\right\rceil}\binom{2n}{n}}{p^n(2n+1)}  \beta^{mn+s}
	= \sqrt{\frac{p}{2}}\,\beta^{s-\frac  m2}\arctan\left(\frac{2\sqrt{\beta^m}}{\sqrt{p^2+16\beta^{2m}}+p}\right)\!,
	\end{gather*}
according to Binet's formulas \eqref{binet}. Other formulas can be proved similarly using corresponding formulas from Theorem \ref{main_thm}. For brevity, we omit their  proofs.  
\end{proof}

As special cases of Theorem \ref{mn+s}, we obtain numerous series involving Fibonacci or Lucas numbers and central binomial coefficients, expressed in terms of the golden ratio. In Examples \ref{Ex2}--\ref{Ex3}, we present only a small collection  of such series. 
\begin{example} \label{Ex2} By setting $m=1$, $s=0$,  and $p=8$ or $p=16$ in \eqref{FL_11}--\eqref{18}, we obtain the following series list:
	\begin{align*}
	&\sum_{n = 0}^\infty \frac{(- 1)^{\left\lceil{\frac{n}{2}}\right\rceil}\binom{2n}{n}}{2^{3n}(2n+1)}  F_{n}= \frac{2\sqrt5}{5\sqrt{\alpha}}
	\big(\arctan C_1  -\alpha\arctanh C_2\big),\\
	&\sum_{n = 0}^\infty \frac{(- 1)^{\left\lceil{\frac{n}{2}}\right\rceil}\binom{2n}{n}}{2^{3n}(2n+1)} L_{n}= \frac{2}{\sqrt{\alpha}}
	\big(\arctan C_1  +\alpha\arctanh C_2\big),\\
	&\sum_{n = 0}^\infty \frac{(- 1)^{\left\lfloor{\frac{n}{2}}\right\rfloor}\binom{2n}{n}}{2^{3n}(2n+1)} F_{n}= \frac{2\sqrt5}{5\sqrt{\alpha}}
	\big(\arctanh C_1  -\alpha\arctan C_2\big),\\
	&\sum_{n = 0}^\infty \frac{(- 1)^{\left\lfloor{\frac{n}{2}}\right\rfloor}\binom{2n}{n}}{2^{3n}(2n+1)} L_{n}= \frac{2}{\sqrt{\alpha}}
	\big(\arctanh C_1  +\alpha\arctan C_2\big),\\
&\sum_{n = 0}^\infty \frac{(- 1)^{\left\lceil{\frac{n}{2}}\right\rceil}\binom{2n}{n}}{2^{4n}(2n+1)} F_{n}= \frac{2\sqrt{10}}{5\sqrt{\alpha}}
	\left(\arctan D_1  -\alpha\arctanh D_2\right),\\
	&\sum_{n = 0}^\infty \frac{(- 1)^{\left\lceil{\frac{n}{2}}\right\rceil}\binom{2n}{n}}{2^{4n}(2n+1)} L_{n}= \frac{2\sqrt{2}}{\sqrt{\alpha}}
	\left(\arctan D_1 + \alpha\arctanh D_2\right),\\
	&\sum_{n = 0}^\infty \frac{(- 1)^{\left\lfloor{\frac{n}{2}}\right\rfloor}\binom{2n}{n}}{2^{4n}(2n+1)} F_{n}=  \frac{2\sqrt{10}}{5\sqrt{\alpha}}
	\big(\arctanh D_1 - \alpha\arctan D_2\big),\\
	&\sum_{n = 0}^\infty \frac{(- 1)^{\left\lfloor{\frac{n}{2}}\right\rfloor}\binom{2n}{n}}{2^{4n}(2n+1)} L_{n}= \frac{2\sqrt{2}}{\sqrt{\alpha}}
	\big(\arctanh D_1 + \alpha\arctan D_2\big),
	\end{align*}
	where 
	\begin{gather*}	C_1=\sqrt{2-2\alpha+\sqrt{9-4\alpha}},\qquad C_2=\sqrt{\sqrt{5+4\alpha}-2\alpha},\\ D_1=\sqrt{4-4\alpha+\sqrt{33-16\alpha}},\qquad D_2=\sqrt{\sqrt{16\alpha+17}- 4\alpha}.
	\end{gather*}
\end{example}
\begin{example}\label{Ex1} By setting $m=1$, $s=0$, and $p=8$ in \eqref{19}--\eqref{22}, we obtain the following series:
	\begin{gather*}
	\sum_{n = 0}^\infty \frac{(- 1)^{\lceil{\frac{n}{2}}\rceil}\binom{2n}{n}}{2^{3n}}  F_{n}
	= \frac{\sqrt{290}}{145}
	\left(A_1-A_2\right)\!,\qquad
	\sum_{n = 0}^\infty \frac{(- 1)^{\lceil{\frac{n}{2}}\rceil}\binom{2n}{n}}{2^{3n}} L_{n}
	= \frac{\sqrt{58}}{29}
	\left(A_1+A_2\right)\!,\\
	\sum_{n = 0}^\infty \frac{(- 1)^{\lfloor{\frac{n}{2}}\rfloor}\binom{2n}{n}}{2^{3n}} F_{n}
	= \frac{\sqrt{290}}{145}
	\left(B_1-B_2\right)\!,\qquad 
	\sum_{n = 0}^\infty \frac{(- 1)^{\lfloor{\frac{n}{2}}\rfloor}\binom{2n}{n}}{2^{3n}} L_{n}
	= \frac{\sqrt{58}}{29}
	\left(B_1+B_2\right)\!,
	\end{gather*}
	where \begin{gather*}
	A_1=\sqrt{\sqrt{29(6-\alpha)}+1-5\alpha},\qquad  A_2=\sqrt{\sqrt{29(5+\alpha)}+5\alpha-4},\\
	B_1=\sqrt{\sqrt{29(6-\alpha)}+5\alpha-1}, \qquad B_2=\sqrt{\sqrt{29(5+\alpha)}+4-5\alpha}.
	\end{gather*}
\end{example}
\begin{example} \label{Ex3} Taking values $m=2$, $s=0$ and $p=16$ in \eqref{FL_11}--\eqref{22} leads to the following series: 
	\begin{gather*} 
	\sum_{n = 0}^\infty \frac{(-1)^{\lfloor{\frac{n}{2}}\rfloor}\binom{2n}{n}}{2^{4n}(2n+1)}  F_{2n}=\frac{2\sqrt{10}}{5\alpha}\big(\arctanh E_1- \alpha^{2}\arctanh E_2\big),\\
	\sum_{n = 0}^\infty \frac{(-1)^{\lfloor\frac{n}{2}\rfloor}\binom{2n}{n}}{2^{4n}(2n+1)} L_{2n}=\frac{2\sqrt{2}}{\alpha}\big(\arctanh E_1+ \alpha^{2}\arctanh E_2\big),\\
	\sum_{n = 0}^\infty \frac{(-1)^{\lceil{\frac{n}{2}}\rceil}\binom{2n}{n}}{2^{4n}(2n+1)}  F_{2n}=\frac{2\sqrt{10}}{5\alpha}\big(\arctan E_1- \alpha^{2}\arctan E_2\big),\\
	\sum_{n = 0}^\infty \frac{(-1)^{\lceil\frac{n}{2}\rceil}\binom{2n}{n}}{2^{4n}(2n+1)}  L_{2n}=\frac{2\sqrt{2}}{\alpha}\big(\arctan E_1+ \alpha^{2}\arctan E_2\big),
	\end{gather*} 
	and
	\begin{gather*} 
	\sum_{n=0}^\infty \frac{(-1)^{\lfloor{\frac{n}{2}}\rfloor}\binom{2n}{n}}{2^{4n}}  F_{2n}=\frac{\sqrt{30}}{15}\left(H_1^+-H_2^+\right)\!,\qquad
	\sum_{n=0}^\infty \frac{(-1)^{\lfloor{\frac{n}{2}}\rfloor}\binom{2n}{n}}{2^{4n}}  L_{2n}=\frac{\sqrt{150}}{15}\left(H_1^++H_2^+\right)\!,\\
	\sum_{n=0}^\infty \frac{(-1)^{\lceil{\frac{n}{2}}\rceil}\binom{2n}{n}}{2^{4n}}  F_{2n}=\frac{\sqrt{30}}{15}\left(H_1^--H_2^-\right)\!,\qquad
	\sum_{n=0}^\infty \frac{(-1)^{\lceil{\frac{n}{2}}\rceil}\binom{2n}{n}}{2^{4n}}  L_{2n}=\frac{\sqrt{150}}{15}\left(H_1^-+H_2^-\right)\!,
	\end{gather*}
	where  
	\begin{gather*}
	E_1=\sqrt{\sqrt{81-48\alpha}-8+4\alpha},\quad  E_2=\sqrt{\sqrt{33+48\alpha}-4\alpha-4},\\ H_1^{\pm}=\frac{\sqrt{\sqrt{18+3\alpha}\mp4}\,(5+\alpha\pm\sqrt{18+3\alpha})}{\sqrt{15+23\alpha}},\quad
	H_2^{\pm}=\frac{\sqrt{\sqrt{21-3\alpha}\mp4}\,(6-\alpha\pm\sqrt{21-3\alpha})}{\sqrt{38-23\alpha}}.
	\end{gather*}
\end{example}

\section{Some additional series}

In this section how additional series can be deduced from the main results from Section 2.
\begin{lemma}\label{lem.ao89abl}
	Let $(f_k)_{k\geq0}$  be a sequence of real or complex numbers. Then
	\begin{gather*}
	\sum\limits_{n = 0}^\infty  {\left( {( - 1)^{\left\lceil {n/2} \right\rceil }  + ( - 1)^{\left\lfloor {n/2} \right\rfloor } } \right)} f_n = 2\sum\limits_{n = 0}^\infty  {( - 1)^n f_{2n}} ,\\
	\sum\limits_{n = 0}^\infty  {\Big( {( - 1)^{\left\lceil {n/2} \right\rceil }  - ( - 1)^{\left\lfloor {n/2} \right\rfloor } } \Big)} f_n  = 2\sum\limits_{n = 0}^\infty  ( - 1)^n f_{2n + 1} .
	\end{gather*}
\end{lemma}
\begin{proof} The proof follows from 
	$$( - 1)^{\left\lceil \frac{n}2 \right\rceil } + ( - 1)^{\left\lfloor \frac{n}2 \right\rfloor } = (-1)^{\frac{n}2}\big(1+(-1)^n\big),\quad
	( - 1)^{\left\lceil\frac{n}2 \right\rceil } - ( - 1)^{\left\lfloor \frac{n}2 \right\rfloor } = (-1)^{\frac{n-1}2}\big(1-(-1)^n\big).$$
\end{proof}
\begin{theorem}\label{thm.xeo97nr}
	If $|x|<1$, then the following identities hold:
	\begin{gather}
	\sum_{n = 0}^\infty   \frac{( - 1)^n \binom{4n}{2n}}{2^{4n}}\, x^n = \frac{1}{\sqrt 2}\frac{\sqrt {1 + \sqrt {1 + x }}}{\sqrt {1 + x}},\nonumber\\
	\sum_{n = 0}^\infty \frac{\binom{4n}{2n}}{2^{4n}}\,x^n   = \frac{1}{{\sqrt 2 }}\frac{{\sqrt {1 + \sqrt {1 - x } } }}{\sqrt {1 - x } },\label{eq.ecfsok7}\\
	\sum_{n = 0}^\infty  \frac{(-1)^n \binom{{2( {2n - 1})}}{{2n - 1}}}{2^{ 2( {2n - 1})}} \, x^n = -\sqrt {\frac{x}{2}} \frac{{\sqrt { - 1 + \sqrt {1 + x} } }}{{\sqrt {1 + x} }}\nonumber,
	\end{gather}
	and
	\begin{equation*}
	\sum_{n = 0}^\infty  \frac{\binom{{2( {2n - 1} )}}{{2n - 1}}}{2^{ 2( {2n - 1})}}\,  x^n  = \sqrt {\frac{x}{2}} \frac{{\sqrt {1 - \sqrt {1 - x} } }}{{\sqrt {1 - x} }}.
	\end{equation*}
\end{theorem}
\begin{proof}
	Addition and subtraction of identities \eqref{eq.e6dkxz3}--\eqref{eq.cu2e27v} while making use of Lemma \ref{lem.ao89abl}. Note also that
	\begin{equation*}
	\sqrt {\frac{{\sqrt {1 + x^2 }  - x}}{{1 + x^2 }}}  + \sqrt {\frac{{\sqrt {1 + x^2 }  + x}}{{1 + x^2 }}}  = \sqrt 2\, \sqrt {\frac{{1 + \sqrt {1 + x^2 } }}{{1 + x^2 }}}.
	\end{equation*}
\end{proof}
\begin{theorem}\label{KA_1}
If $r$ is an even integer, then
\begin{gather}
\sum_{n = 0}^\infty  {\frac{{\binom{{4n}}{{2n}}}}{{2^{4n} }}\frac{{L_{rn} }}{{L_r^n }}} = \frac{L_r^{1/4}}{\sqrt 2} \left(L_r^{3/2}+ L_{r/2}  + 2\left( 1+ L_r + L_r^{1/2}L_{r/2} \right)^{1/2}\right)^{1/2},\label{eq.fw15prg}\\
\sum_{n = 0}^\infty  {\frac{{\binom{{4n}}{{2n}}}}{{2^{2n}L_r^{2n} }}} = \frac{\sqrt {5\alpha ^r L_r }}{5F_r} ,\quad r\ne 0\label{eq.b3tts0w},
\end{gather}
and
\begin{equation}\label{eq.wo12uea}
\sum_{n = 0}^\infty  {\frac{{\binom{{4n}}{{2n}}}}{20^{n}}}  = \sqrt {\alpha \sqrt 5 } .
\end{equation}
\end{theorem}
\begin{proof}
Set $x=\alpha^r/L_r$ and $x=\beta^r/L_r$, in turn, in \eqref{eq.ecfsok7} to obtain \eqref{eq.fw15prg}. Now, express \eqref{eq.ecfsok7} as
\begin{equation}\label{eq.r84jsze}
\sum_{n = 0}^\infty  {\frac{{\binom{{4n}}{{2n}}}}{{2^{4n} }}\left( {1 - x^2 } \right)^n}  = \frac{{\sqrt {1 + x} }}{\sqrt 2\,x},\quad 0< x<1. 
\end{equation}
Set $x=\sqrt 5\,F_r/L_r$ to get \eqref{eq.b3tts0w} and $x=1/\sqrt 5$ to obtain \eqref{eq.wo12uea}.
\end{proof}
\begin{theorem} \label{KA_2}
Let $H_n$ be the $n$-th harmonic number defined by
\begin{equation*}
H_n=\sum_{k=1}^n\frac1k,\quad H_0=0.
\end{equation*}
Then
\begin{equation}\label{eq.i5qptp9}
\sum_{n = 0}^\infty  {\frac{{\binom{{4n}}{{2n}}}}{{2^{4n} }}\frac{{H_{n + 1} }}{{n + 1}}}  =  \frac{80}{9}-\frac{32\sqrt2}{9}- \frac{{8\sqrt 2 }}{3}\ln\! \Big( {\frac{{1 + \sqrt 2 }}{2}} \Big)\!.
\end{equation}
\end{theorem}
\begin{proof}
From \eqref{eq.r84jsze} we get
\begin{equation*}
\sum_{n = 0}^\infty  {\frac{{\binom{{4n}}{{2n}}}}{{2^{4n} }}\int_0^1 {x\left( {1 - x^2 } \right)^n \ln x\,dx} }  = \frac{1}{{\sqrt 2 }}\int_0^1 {\sqrt {1 + x} \ln x\,dx} ,
\end{equation*}
and hence \eqref{eq.i5qptp9}, since
\begin{align*}
\int {\sqrt {1 + x} \ln x\,dx}  &= \frac{2}{3}\left( {( 1 + x)\sqrt {1 + x}  - 1} \right)\ln x-\frac{4}{9}( 4 + x)\sqrt {1 + x}\\&\,\quad +\frac{4}{3}\ln\left( 1 + \sqrt {1 + x} \right),
\end{align*}
and a simple change of variable in the well-known integration formula
\begin{equation*}
\int_0^1 {\left( {1 - x} \right)^{v - 1} \ln x\,dx}  =  - \frac{{H_v }}{v},\quad v\geq1,
\end{equation*}
gives
\begin{equation*}
\int_0^1 {x\left({1-x^2}\right)^{v-1}\ln x\,dx} = -\frac{{H_v }}{4v},\quad v\geq1.
\end{equation*}
\end{proof}
\begin{remark}
Similar results to Theorems \ref{KA_1} and \ref{KA_2} can be derived from the final identity in Theorem \ref{thm.xeo97nr}; we leave these derivations to the reader.
\end{remark}

\section{Conclusion}

This paper examined several series involving central binomial coefficients and explored their structural properties and connections with other mathematical objects. By deriving new closed-form representations for these series, we have contributed to a deeper understanding of the alternating patterns in series that involve central binomial coefficients.

Furthermore, we extended our exploration to series involving Fibonacci and Lucas numbers, unveiling new identities and highlighting their potential applications in various mathematical fields. These results enhance our understanding of the interplay between central binomial coefficients and other mathematical sequences and offer a foundation for future research in this area.

Future work could focus on exploring additional connections between these series and other special functions, as well as extending these identities to higher-order binomial coefficients and related sequences.


\begin{thebibliography}{99}

{\small
\bibitem{AFG1}
K.~Adegoke, R.~Frontczak and T.~Goy, Fibonacci--Catalan series, \emph{Integers}
{\bf 22} (2022), \#A110.
\vspace{-2mm}

\bibitem{AFG2}
K.~Adegoke, R.~Frontczak and T.~Goy, On a family of infinite series
with reciprocal Catalan numbers, \emph{Axioms} {\bf 11} (2022), Article 165.
\vspace{-2mm}

\bibitem{NNTDM-24}
K.~Adegoke, R.~Frontczak and T.~Goy, Evaluation of some alternating series involving the binomial coefficients $C(4n,2n)$,  arxiv preprint arXiv:2404.05770 [math.NT], 2024. Available at {https://arxiv.org/abs/2404.05770}.
\vspace{-2mm}

\bibitem{alzer}
H. Alzer and G.~V. Nagy, Some identities involving central binomial coefficients and Catalan numbers, \textit{Integers} 20 (2020), \#A59.
\vspace{-2mm}

\bibitem{banderier}
C.~Banderier and M.~Wallner,
Young tableaux with periodic walls:
counting with the density method, \textit{S\'em. Lothar. Combin. }\textbf{85B} (2021), Article \#47.
\vspace{-2mm}

\bibitem{barry}
P.~Barry, On the central antecedents of integer (and other) sequences, \textit{J. Integer Seq.}  \textbf{23} (2020), Article 20.8.3.
\vspace{-2mm}

\bibitem{batir1}
N.~Batir,   H.~K\"uc\"uk and S.~Sorgun, Convolution identities involving the central binomial coefficients and Catalan numbers, \textit{Trans. Comb.} {\bf 11}(4) (2021), 225--238.
\vspace{-2mm} 

\bibitem{batir2}
N.~Batir and A.~Sofo, Finite sums involving reciprocals of the binomial and central binomial coefficients and harmonic numbers, \textit{Symmetry} {\bf 13}(11) (2021), 2002. 
\vspace{-2mm}

\bibitem{bhandari}
N.~Bhandari, Infinite series associated with the ratio and product of central binomial coefficients,
\emph{J. Integer Seq.} {\bf 25} (2022), Article 22.6.5.
\vspace{-2mm}

\bibitem{boya}
K.~N. Boyadzhiev, Series with central binomial coefficients, Catalan numbers, and harmonic numbers, \emph{J. Integer Seq.} {\bf 15} (2012), Article 12.1.7. 
\vspace{-2mm}

\bibitem{campbell}
J.~M. Campbell, Series containing squared central binomial coefficients and alternating harmonic numbers, \emph{Mediterr. J. Math.} {\bf 16} (2019), Article 37.
\vspace{-2mm}

\bibitem{chen}
H.~Chen, Interesting Ramanujan-like series associated with powers of central binomial coefficients,
\emph{J. Integer Seq.} {\bf 25} (2022), Article 22.1.8.
\vspace{-2mm}

\bibitem{chu}
W.~Chu and F.~L.~Esposito, Evaluation of Ap\'{e}ry-like series through multisection method, \emph{J. Class. Anal.}
{\bf 12} (2018), 55--81.
\vspace{-2mm}

\bibitem{FanChu}
Z.~Fan and W.~Chu, Alternating series in terms of Riemann zeta function and Dirichlet beta function, \emph{Electron. Res. Arch.}
{\bf32}(2) (2024), 1227--1238.
\vspace{-2mm}

\bibitem{ferrari}
L.~Ferrari, Some combinatorics related to central binomial coefficients: Grand-Dyck paths, coloured noncrossing partitions and signed pattern avoiding permutations, \textit{Graphs Combin.}  \textbf{26}(1) (2010), 51--70. 
\vspace{-2mm}

\bibitem{ford}
K.~Ford and S.~Konyagin, Divisibility of the central binomial coefficient $\binom{2n}{n}$, \textit{Trans. Amer. Math. Soc}. \textbf{374} (2021), 923--953. 
\vspace{-2mm}

\bibitem{gould}
H.~W. Gould, Combinatorial Identities:  A Standardized Set of Tables Listing 500 Binomial Coefficient Summations, Morgantowm, USA, 1972.
\vspace{-2mm}

\bibitem{graham}
R.~L. Graham, D.~E. Knuth and O.~Patashnik, Concrete Mathematics: A Foundation for Computer Science, Addison-Wesley, Reading, USA, 1994.
\vspace{-2mm}

\bibitem{kessler}
D.~A. Kessler and J. Schiff, The asymptotics of factorials, binomial coefficients and Catalan numbers, \textit{J. Integer Seq.} \textbf{24} (2021), Article 21.8.3.
\vspace{-2mm}

\bibitem{lehmer}
D.~H.~Lehmer, Interesting series involving the central binomial coefficient, \emph{Amer. Math. Monthly} {\bf 92} (1985), 449--457.
\vspace{-2mm}

\bibitem{li1}
N.~N.~Li and W. Chen, Infinite series containing central binomial coefficients, \emph{Bull. Math. Soc. Sci. Math. Roumanie} {\bf 66(114)}(4) (2023), 363--372.
\vspace{-2mm}

\bibitem{li2}
C.~Li and W. Chen, Series involving cubic central binomial coefficients of convergence rate $1/64$, \emph{Bull. Malays. Math. Sci. Soc.} {\bf 47} (2024), Article 94.
\vspace{-2mm}

\bibitem{lim}
D.~Lim and A. K. Rathie, A note on two known sum involving the central binomial coefficients with an application, \emph{J. Korean Soc. Math. Educ. Ser. B Pure Appl. Math.} {\bf 29}(2) (2022), 171--177.
\vspace{-2mm}

\bibitem{Mikic}
J.~Miki\'{c}, A proof of a famous identity concerning the convolution of the central binomial coefficients,
\textit{J. Integer Seq.} {\bf 19} (2016), Article 16.6.6.
\vspace{-2mm}

\bibitem{pomer}
C.~Pomerance, Divisors of the middle binomial coefficient, \emph{Amer. Math. Monthly} \textbf{122}(7) (2014), 636--644.
\vspace{-2mm}

\bibitem{qi}
F.~Qi, C.-P.~Chen and D.~Lim, Several identities containing central binomial
coefficients and derived from series expansions of powers of the arcsine function, \emph{Results Nonlinear Anal.} \textbf{4}(1) (2021), 57--65.
\vspace{-2mm}

\bibitem{OEIS}
N.~J.~A. Sloane (ed.), {The On-Line Encyclopedia of Integer Sequences}. Published electronically at https://oeis.org, 2025.
\vspace{-2mm}

\bibitem{sprug}
R.~Sprugnoli, Sums of reciprocals of the central binomial coefficients, {\em Integers} {\bf 6} (2006), \#A27.
\vspace{-5mm}

\bibitem{tauraso}
R.~Tauraso, Some congruences for central binomial sums
involving Fibonacci and Lucas numbers, \emph{J. Integer Seq.} {\bf 19} (2016), Article 16.5.4.

}
\end{thebibliography}
\end{document}